\documentclass[a4paper,twoside,11pt]{article}
\usepackage[english]{babel}
\usepackage[T1]{fontenc}
\usepackage[ansinew]{inputenc}
\usepackage{geometry}
\usepackage{color} 
\usepackage{lmodern}
\usepackage{graphicx}
\usepackage{amsmath} \numberwithin{equation}{section}
\usepackage{amsthm}
\usepackage{amsfonts}
\usepackage{amssymb}
\usepackage{graphics}
\usepackage[pdfstartview=FitH]{hyperref}
\usepackage{enumerate}
\usepackage{dsfont}

\theoremstyle{plain}
\newtheorem{theorem}{Theorem}[section]

\newtheorem{question}[theorem]{Question} 
\newtheorem{remark}[theorem]{Remark}\newtheorem{example}[theorem]{Example}
\theoremstyle{definition}
\newtheorem{definition}[theorem]{Definition}
\newtheorem{proposition}[theorem]{Proposition}    
\theoremstyle{remark}

\newcommand\C{\mathbb C}      
\newcommand\R{\mathbb R}        
         
\newcommand\N{\mathbb N}        
\newcommand\Ha{\mathbb H}       
\newcommand\D{\mathbb D}     
\renewcommand\Im{\text{Im}}

\newcommand{\Landauo}{{\scriptstyle\mathcal{O}}}

\newcommand{\supp}{\operatorname{supp}}

\begin{document}
\parindent 0pt 

\setcounter{section}{0}

\title{On infinite-slit limits of multiple SLE}
\date{\today}
\author{
 Sebastian Schlei{\ss}inger  \thanks{Supported by the ERC grant  ``HEVO - Holomorphic Evolution Equations'' no. 277691.}}

\title{The Chordal Loewner Equation and Monotone Probability Theory}
\maketitle

\abstract{In \cite{MR2053711}, O. Bauer interpreted the chordal Loewner equation in terms of non-commutative probability theory. We follow this perspective and identify the chordal Loewner equations as the non-autonomous versions of evolution equations for semigroups in monotone and anti-monotone probability theory. We also look at the corresponding equation for free probability theory.}\\


\tableofcontents
 
\parindent 0pt

\section{Introduction}

Denote by $\Ha=\{z\in\C \,|\, \Im(z)>0\}$ the upper half-plane. 
Let $(\nu_t)_{t\geq0}$ be a family of probability measures on $\R.$ The chordal (ordinary) Loewner equations are given by

\begin{equation}\label{eq1}
\frac{\partial}{\partial t}g_t  = \int_\R \frac1{g_t-u} \,\nu_t(du) \quad \text{for almost every $t\in[0, \infty)$}, \quad g_0(z)=z\in\Ha,\\
\end{equation}
\begin{equation}\label{eq2}
\frac{\partial}{\partial t}\varphi_t  = \int_\R \frac1{u-\varphi_t} \,\nu_t(du)  \quad \text{for almost every $t\in[0, \infty)$}, \quad \varphi_0(z)=z\in\Ha.
\end{equation}

In the first case, the mappings $z\mapsto g_t(z)$ are conformal mappings from $\Ha\setminus K_t$ onto $\Ha,$ where $(K_t)_{t\geq0}$ is a family of growing hulls, i.e. $K_t\subset \Ha,$ $\Ha\setminus K_t$ is simply connected and $K_s\subset K_t$ whenever $s\leq t.$ The initial condition implies $K_0=\emptyset.$
The second equation is interpreted in a similar way.\\

In this note, we show how these equations can be interpreted in terms of monotone probability theory (equation \eqref{eq2}) and anti-monotone probability theory (equation \eqref{eq1}). \\
These relations are in fact quite simple. In case of the second equation \eqref{eq2}, we have that $1/\varphi_{t}$ is the Cauchy transform of a probability measure $\mu_t$. The process $(\mu_t)_{t\geq0}$, in turn, can be interpreted to describe a ``quantum process'' $(X_t)_{t\geq0}$, which can be seen as a collection of 
self-adjoint linear operators with monotonically independent increments such that the distribution of $X_t$ is given by $(\mu_t)_{t\geq0}.$\\
 
In what follows, we explain this connection in more detail. We also take a look at the corresponding differential equation in free probability theory. 

\section{Non-commutative probability}\label{multipleSLE}

Non-commutative probability theory provides an abstract description of random variables, motivated by the role that observables play in quantum mechanics.\\
In the following, we recall some of the basic notions of free probability theory and monotone probability theory. Both are non-commutative probability theories in which the classical notion of \emph{independent} random variables is replaced by \emph{freely independent/monotonically independent} random variables. We refer to \cite{MR2132092} for an introduction.\\ 

A \emph{non-commutative probability space} $(\mathcal{A}, \varphi)$ consists of a unital algebra $\mathcal{A}$ and a linear functional $\varphi:\mathcal{A}\to\C$ with $\varphi(\mathds{1})=1.$ The elements of $\mathcal{A}$ are called random variables and $\varphi$ should be thought of as an expectation. The distribution of a random variable $a$ is simply defined as the collection of all moments $\varphi(a^k),$ $k\in\N.$\\

Furthermore, $(\mathcal{A}, \varphi)$ is called \emph{$C^*$--probability space} if $\mathcal{A}$ is a $C^*$--algebra and $\varphi$ is a state, i.e. a positive linear functional of norm 1.

\begin{example}\label{matrices} Let $\mathcal{A}$ be the space of all $N\times N$--matrices with the spectral norm and let $\varphi(a)=\frac1{N}\operatorname{Tr}(a).$  Then $(\mathcal{A}, \varphi)$ is a $C^*$--probability space.
\end{example}

\begin{example} Let $H$ be a Hilbert space and let $\mathcal{A}$ be the space ${\mathcal B}(H)$ of all bounded linear operators on $H.$ Furthermore, let $\Omega \in H$ be a unit vector and define $\varphi$ by $\varphi(a) = \left< a \Omega, \Omega\right>.$ Then $(\mathcal{A}, \varphi)$ is a $C^*$--probability space.
\end{example}

In the following, we assume that $(\mathcal{A},\varphi)$ is a $C^*$--probability space.\\
If $a\in \mathcal{A}$ is self-adjoint, then the distribution of $a$ as defined above can be identified with a probability measure on $\R$ by using the spectral theorem: There exists a probability measure on $\R$ (supported on the spectrum $\sigma(a)$) such that for every polynomial $p$, the value $\varphi(p(a))$ can be represented by
$$\varphi(p(a)) = \int_{\sigma(a)}p(z) \, \mu(dz).$$

The measure $\mu$ has compact support. However, one can generalize the setting of $C^*$--probability spaces to deal also with unbounded self-adjoint random variables. In Sections \ref{up} and \ref{uq}, we  stick to the setting of $C^*$--probability spaces for the sake of simplicity.

\subsection{Free probability theory}\label{up}

Free probability theory has been introduced by D. Voiculescu in \cite{MR799593}. It is based on a 
non-commutative notion of independence of random variables, the \emph{free independence}.\\

A collection $a_1,a_2,...,a_N\in \mathcal{A}$ of random variables is called  \emph{freely independent} if 
$$\varphi(p_1(a_{j(1)})...p_k(a_{j(k)}))=0$$ for all polynomials $p_1,...,p_k$ such that $j(1)\not=j(2)\not=...\not=j(k)$ and 
$\varphi(p_i(a_{j(i)}))=0$ for all $i=1,...,k.$
To simplify the notation later on, we will call an $N$--tuple $(a_1,a_2,...,a_N)\in \mathcal{A}^N$  \emph{freely independent} if $a_1,...,a_N$ are freely independent.\\

The usefulness of the above definition is due to the following fact: Let $a,b\in\mathcal{A}$ be two freely independent random variables. Then the moments of $a+b$ can be calculated by using the moments of $a$ and $b$ only (\cite[Proposition 4.3]{MR799593}). This leads to the free convolution:\\
Assume that $a, b$ are freely independent and self-adjoint random variables with distributions $\mu$ and $\nu$. The distribution of $a+b$, denoted by $\mu\boxplus\nu$, is called the \emph{free convolution} of $\mu$ and $\nu.$

The free convolution of probability measures is closely related to their Cauchy transforms:\\
First, the Cauchy transform (or Stieltjes transform) is given by $$G_{\mu}(z)=\int_\R\frac{1}{z-x}\,\mu(dx), \qquad z\in\C\setminus \R.$$
 (Note that the measure $\mu$ can be recovered from $G_{\mu}$ by the Stieltjes-Perron inversion formula, see \cite[Theorem F.2]{MR2953553}.) \\
Now we define $V_\mu$ as the right inverse of $G_\mu$, i.e. $V_\mu$ is the solution of $G_\mu(V_\mu(z))=z$, $z\in\Ha,$ with $V_\mu(z)\sim \frac1{z}$ for  $z$ near $0.$ (For probability measures with unbounded support,
$V_\mu$ exists as a holomorphic function defined on a Stolz angle near $0$; see \cite[Section 5]{MR1254116}.)\\
Finally, the $R$--transform $R_{\mu}(z)$ of $\mu$ is defined by $R_{\mu}(z)=V_\mu(z)-\frac1{z}.$\\

 For probability measures $\mu$ and $\nu$ on $\R,$ the free convolution $\mu \boxplus \nu$ can be calculated by the formula
$$R_{\mu\boxplus\nu}(z) = R_\mu(z) + R_\nu(z).$$

\begin{example} Free probability theory possesses a non-commutative analogue of the central limit theorem; see \cite{MR799593}. The free analogue of the normal distribution (with mean zero and variance $\sigma^2$) is given by Wigner's semicircle distribution $\mu_{W,\sigma^2}$ given by $$\frac{\sqrt{4\sigma^2-x^2}}{\pi 2\sigma^2} dx, \quad x\in[-2\sigma, 2\sigma].$$
Here, we have $G_{\mu_{W,\sigma^2}}(z) = \frac{2}{z+\sqrt{z^2-4\sigma^2}},$ and $R_{\mu_{W,\sigma^2}}(z)=\sigma^2 z,$ and consequently, the semicircle distribution is freely stable: 
$$\mu_{W,\sigma^2} \boxplus \mu_{W,\tau^2} = \mu_{W,\sigma^2+\tau^2}.$$
\end{example}

\begin{example} Let $(\mu_t)_{t\geq0}$ be a semigroup with respect to free convolution, i.e. $\mu_{t+s}=\mu_t \boxplus \mu_s,$ such that 
$\mu_0=\delta_0.$ Let $\alpha$ be an arbitrary probability measure and define $G_t:=G_{\mu_t\boxplus \alpha}.$ Then $G_t$ satisfies the PDE
\begin{equation}\label{semigroup}
\frac{\partial}{\partial t}G_t(z) = -\frac{\partial}{\partial z}G_t(z)\cdot R(G_t(z)), \quad G_0(z) = G_\alpha, \end{equation}
where $\Im(z)<0$ and $R$ is the $R$-transform $R=R_{\mu_1};$ see \cite[p.74]{MR1228526}. In this case, $R$ has an analytic extension to $\Ha$ and $R:\Ha\to\overline{\Ha}.$\\

For example, take $\mu_{t}=\mu_{W,t}$, which leads to
\begin{equation}\label{heat}
\frac{\partial}{\partial t}G_t(z) = -\frac{\partial}{\partial z}G_t(z)\cdot G_t(z), \quad G_0(z) = G_\alpha,
 \end{equation}
i.e. the ``free analogue of the heat equation is the complex inviscid Burgers equation'' (\cite{Novak_lectures}, p.44), because a realization of the process $\mu_{W,t}$ is called a free Brownian motion, see Section \ref{realization}.
\end{example}

\subsection{Monotone independence}\label{uq}

For a probability measure $\mu$ we define the 
$F$--transform of $\mu$ simply as $F_\mu:\Ha\to\Ha, F_\mu(z) := 1/G_\mu(z).$

\begin{remark}Let $\mu, \nu$ be probability measures. Then there exist holomorphic mappings $\omega_1, \omega_2:\Ha\to\Ha$ such that
$$F_{\mu\boxplus \nu}(z) = F_{\mu}(\omega_1(z)) = F_{\nu}(\omega_2(z)) \qquad \text{for all $z\in \Ha.$} $$
Furthermore, also $\omega_1, \omega_2$ have the form $\omega_1 = F_{\sigma_1}$ and $\omega_2 = F_{\sigma_2}$ for probability measures $\sigma_1, \sigma_2;$ see \cite[Theorem 3.1]{MR1605393}.
\end{remark}

Now one defines the \emph{monotone convolution} $\mu \rhd \nu$ by 
$$F_{\mu \rhd \nu} = F_\mu \circ F_\nu.$$
This convolution is related to another notion of independence of random variables, 
the \emph{monotone independence}, which was introduced by N. Muraki 
(\cite{MR1467953}) and independently by De. Giosa, Lu
(\cite{MR1483010, MR1455615}). \\

Let $a_1,...,a_N\in \mathcal{A}.$ The tuple $(a_1,a_2,...,a_N)$, as an ordered collection of random variables, is called \emph{monotonically independent} if
$$\varphi(a_{i_1}^{p_1}\dots a_{i_k}^{p_k} \dots 
a_{i_m}^{p_m})=\varphi(a_{i_k}^{p_k})\cdot 
\varphi(a_{i_1}^{p_1}\dots a_{i_{k-1}}^{p_{k-1}} a_{i_{k+1}}^{p_{k+1}} \dots 
a_{i_m}^{p_m})$$
whenever $i_{k-1}<i_k>i_{k+1}$ (one of the inequalities is eliminated when $k=1$ or $k=m$); see \cite[Section 2]{hasebe2}.
\begin{remark} In \cite{MR1853184}, Muraki defines monotone independence by the stronger conditions 
\begin{itemize}
\item[(a)] $a_i a_j^p a_k = \varphi(a_j^p) a_i a_k $ whenever $i<j$ and $j>k.$
\item[(b)]\begin{eqnarray*}&&\varphi(a_{i_m}^{p_m}\dots a_{i_2}^{p_2}
a_{i_1}^{p_1}a_{i}^{p}a_{j_1}^{q_1}a_{j_2}^{q_2}\dots 
a_{j_n}^{q_n})\\&=&\varphi(a_{i_m}^{p_m})\dots \varphi(a_{i_2}^{p_2})\varphi(
a_{i_1}^{p_1})\varphi(a_{i}^{p})\varphi(a_{j_1}^{q_1})\varphi(a_{j_2}^{q_2})\dots \varphi(a_{j_n}^{q_n})
\end{eqnarray*} whenever $i_m>...>i_2>i_1>i<j_1<j_2<...<j_n.$
\end{itemize}
\end{remark}

\begin{remark}
Furthermore, there are also several other notions of independence and convolutions. Some interesting relations and decompositions for convolutions are studied in \cite{MR2321046}.\\
Finally, we note that N. Muraki showed in \cite{MR2016316} that there are only five ``nice'', so called natural independences: the tensor,
free, Boolean, monotone and anti-monotone independence; see also \cite[p.198]{MR2213451}.
\end{remark}

Let $a,b \in \mathcal{A}$ be self-adjoint random variables such that $(a,b)$ is monotonically independent. Denote by $\mu$ and $\nu$ the probability measures of $a$ and $b$ respectively. Then we have: the distribution of $a+b$ is exactly the measure $\mu \rhd \nu
$, see \cite[Theorem 4]{MR1853184}.

\begin{example}\label{arcsine} Let $\mu_{A,\sigma^2}$ be the arcsine distribution with mean 0 and variance $\sigma^2,$ i.e. $\frac1{\pi\sqrt{2\sigma^2-x^2}} dx,$ $x\in[-\sqrt{2}\sigma, \sqrt{2}\sigma].$ Then $F_{\mu_{A,\sigma^2}}(z)=\sqrt{z^2-2\sigma^2}$ and we have
$$\mu_{A,\sigma^2} \rhd \mu_{A,\tau^2} = \mu_{A,\sigma^2+\tau^2}.$$
\end{example}

The arcsine distribution plays the role of the Wigner semicircle distribution in free probability;  see \cite[Theorem 2]{MR1853184} for a central limit theorem in monotone probability theory.

\section{Non-autonomous evolution equations}\label{st}

In this section, probability measures are not assumed to have bounded support. We note that the Cauchy, 
F- and R-transform, as well as the free and (anti-)monotone convolutions are also defined for this general case by the same formulas.

\subsection{The chordal Loewner equation}

In \cite{MR1512136}, C. Loewner introduced a differential equation for conformal mappings to attack the so called Bieberbach conjecture: Let $\D\subset\C$ be the unit disc and assume that $f:\D\to\C$ is univalent (=holomorphic and injective) with $f(0)=0$ and $f'(0)=1.$ Let $a_n$ be the coefficients of the power series expansion $f(z)=z + \sum_{n\geq 2}a_nz^n$. Then $$|a_n|\leq n.$$ Loewner could prove this inequality for $n=3$ and the conjecture has been proven completely in 1985 by L. de Branges. Since its introduction, Loewner's approach has been extended and the Loewner differential equations are now an important tool in the theory of conformal mappings. In the following, we describe a special differential equation that goes back to P. Kufarev. We refer to \cite{AbateBracci:2010}
for an historical overview of Loewner theory. \\

The so called chordal Loewner equation can be described as follows:\\
Take a family $\{\nu_t\}_{t\geq 0}$ of probability measures and let $M_t(z)=G_{\nu_t}(z)=\int_\R\frac1{z-x}\,\nu_t(dx)$ be the Cauchy transform of $\nu_t.$ Assume that 
\begin{equation}\label{regularity}
\text{$t \mapsto M_t(z)$ is measurable for every $z\in\Ha$.}
\end{equation}

The chordal Loewner equation is given by the Carath\'{e}odory ODE (``a.e.'' stands for ``almost every'')
\begin{equation}\label{ODE1}\frac{\partial}{\partial t}g_t  = M_t(g_t) \quad \text{for a.e. $t\in[0, \infty)$}, \quad g_0(z)=z\in\Ha,
\end{equation}
and has a unique solution (\cite[Theorem 4]{MR1201130}). \\

For fixed $z\in\Ha$, the solution $t\mapsto g_t(z)$ may have a finite lifetime $T(z)>0$ in the sense that $g_t(z)\in\Ha$ for all $t<T(z)$, but $\lim_{t\nearrow T(z)}\Im(g_t(z))=0.$ \\
If we fix a time $t>0$ and let $K_t=\{z\in\Ha \,|\, T(z)\leq t\},$ then $g_t(z)$  is interpreted as the conformal mapping 
$$\text{$g_t:\Ha\setminus K_t\to \Ha$\quad with the normalization \quad$g_t(z) = z + \frac{t}{z} + \Landauo(1/z)$}$$
 as $z\to \infty$ non-tangentially in $\Ha$. The sets $K_t\subset \Ha$ are growing \emph{hulls}, which means $\Ha\setminus K_t$ is simply connected and $K_s\subset K_t$ whenever $s\leq t.$ As we start with the identity mapping, we have $K_0=\emptyset.$

\begin{figure}[ht]
\rule{0pt}{0pt}
\centering
\includegraphics[width=11cm]{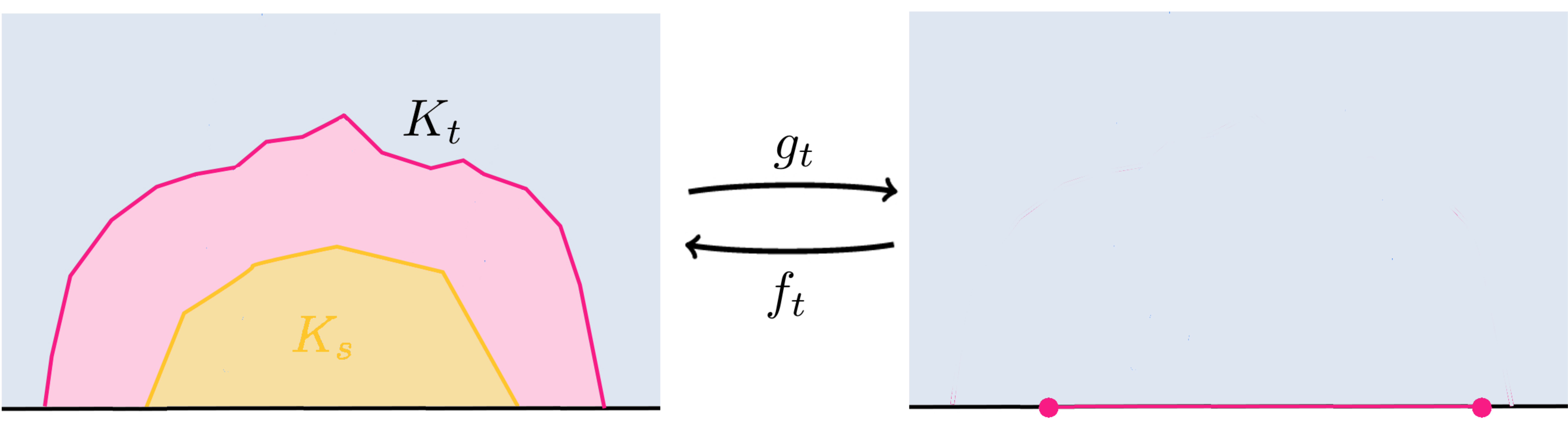}
\caption{The mappings $g_t$ and $f_t$.}
\label{fig1}
\end{figure}

The inverse mappings $f_t=g_t^{-1}$ satisfy the Loewner PDE

\begin{equation}\label{PDE1}\frac{\partial}{\partial t}f_t  = - \frac{\partial}{\partial z}f_t \cdot  M_t(z), \quad f_0(z)=z\in\Ha.
\end{equation}

Now, as noted in \cite{MR2053711}, one can now consider the mapping $G_t:=1/f_t$, which is the Cauchy transform of a measure $\mu_t,$ and so the Loewner equation can be interpreted as a mapping
$$\mathcal{L}: \quad \{\nu_t\}_{t\geq0} \mapsto \{\mu_t\}_{t\geq0}.$$

\begin{example}\label{line_segment} Let $\nu_t=\delta_0$ for all $t\ge0.$ The solution $g_t$ of the Loewner equation 
\begin{equation}\label{LoewnerPDE}\frac{\partial}{\partial t}g_t  = \frac1{g_t}, \quad g_0(z)=z\in\Ha,
\end{equation}
is given by $g_t(z)=\sqrt{z^2+2t}.$ Hence, the hull $K_t$ is a straight line segment connecting $0$ to $\sqrt{2t}i.$\footnote{
  As $t\to\infty$, we obtain a straight line segment $L$ connecting $0$ and $\infty.$ The set $\R\cup\{\infty\}$, as a subset of the Riemann sphere, is a circle. Hence, $L$ is a \emph{chord} of $\R\cup\{\infty\}$. In this sense, the adjective ``chordal'' in ``chordal Loewner equation'' suggests that we are connecting points of the real line with $\infty$ by the ``chord'' $\cup_{t\geq0}K_t$.} We have $G_t=\frac1{\sqrt{z^2-2t}}$ and thus the measure $\mu_t$ is the arcsine distribution with variance $t$; see Example \ref{arcsine}.
\end{example}

The family $G_t$ can be characterized by the differential equation
\begin{equation}\label{Loewner}
\frac{\partial}{\partial t}G_t(z) = -\frac{\partial}{\partial z}G_t(z) \cdot M_t(z), \quad G_0(z)=\frac1{z}.
\end{equation}

\begin{example}\label{fix} As noted in \cite{MR2053711}, the only fixed point of $\mathcal{L}$, i.e. $M_t=G_t$ for all $t,$ is given by $\mu_t=\mu_{W,t},$ which follows by taking $\alpha=\delta_0$ in equation \eqref{heat} and comparing it with equation \eqref{Loewner}.
\end{example}

\begin{example}If $M_t$ does not depend on $t,$ i.e. 
\begin{equation}\label{Loewner_semi}
\frac{\partial}{\partial t}f_t(z)  = - \frac{\partial}{\partial z}f_t(z) \cdot  M_0(z), \quad f_0(z)=z,
\end{equation} then the mappings $f_t=1/G_t$ form a semigroup with respect to composition: $f_{t+s} = f_t \circ f_s=f_{s}\circ f_t$.
\end{example}

From this last example we see that both \eqref{semigroup} and \eqref{Loewner_semi} describe semigroups with respect to different convolutions. In \eqref{semigroup} we have
$$ \mu_{t+s} = \mu_t \boxplus \mu_s, $$
while $\mu_t$ in \eqref{Loewner_semi} satisfies
$$ \mu_{t+s} = \mu_t \rhd \mu_s,$$
because  $f_t=F_{\mu_t}.$\\

Next we look at the non-autonomous versions of these equations from the perspective of monotone, anti-monotone and free probability theory.

\subsection{Monotone evolution families}

A (one-real-parameter) monotone semigroup $(\mu_{t})_{t\geq0}$ is a family of probability measures having the property $\mu_{t+s}=\mu_t \rhd \mu_s,$ $\mu_0=\delta_0.$ Now we generalize monotone semigroups to monotone evolution families.

\begin{definition}\label{LME}
We call a collection $(\sigma_{s,t})_{0\leq s\leq t}$ of probability measures a \emph{monotone evolution family} if it satisfies the conditions
\begin{itemize}
\item[(a)] $\sigma_{t,t}=\delta_0$,
\item[(b)] $\sigma_{s,t} = \sigma_{u,t} \rhd \sigma_{s,u}$ whenever $0\leq s\leq u\leq t$,
\item[(c)] $\sigma_{s,u}$ converges weakly to $\sigma_{s,t}$ as $u\to t$.
\end{itemize}
In addition, it is called \emph{normal} if the first and second moments exist and
\begin{itemize}
\item[(d)] $\int_\R x \sigma_{s,t}(dx)=0$ and $\int_\R x^2 \sigma_{s,t}(dx)=t-s$ for all $0\leq s \leq t.$
\end{itemize}
\end{definition}

 Let $(\nu_t)_{t\geq0}$ be a family of probability measures such that the Cauchy transforms $M_t=\int_\R\frac1{z-u}\,\nu_t(du)$ satisfy \eqref{regularity}, and consider the ``time reversed'' version of \eqref{ODE1}:

\begin{equation}
\label{ODE2}
\frac{\partial}{\partial t}\varphi_{s,t} = -M_t(\varphi_{s,t}) \qquad \text{for a.e. $t\in[s, \infty)$},\quad \varphi_{s,s}(z)=z\in\Ha.
\end{equation}
\begin{remark} Fix some $T>0.$ Let $h_t$ be the solution to $\frac{\partial}{\partial t}h_{t} = M_{T-t}(h_t).$ Then $h_T$ is the inverse of $\varphi_{0,T}.$ 
\end{remark}

According to \cite[Theorem 4]{MR1201130}, \eqref{ODE2} has a unique solution $\varphi_{s,t}:\Ha\to \Ha$, which is an evolution family of holomorphic mappings in the sense that 
\begin{equation}\label{evol}\varphi_{t,t}(z)=z, \qquad  \varphi_{s,t} = \varphi_{u,t} \circ \varphi_{s,u} \quad \text{whenever $0\leq s\leq u\leq t$.} \end{equation}

These solutions are exactly the $F$-transforms of normal monotone evolution families.

\begin{theorem}\label{Sonntag}
Let $(\nu_t)_{t\geq0}$ and $\varphi_{s,t}$ be defined as above. For each $0\leq s\leq t$, the mapping $\varphi_{s,t}$ is the $F$--transform of a probability measure $\sigma_{s,t},$ and  $(\sigma_{s,t})_{0\leq s\leq t}$ is a normal monotone evolution family.\\

Conversely, let $(\sigma_{s,t})_{0\leq s\leq t}$ be a normal monotone evolution family and let $\varphi_{s,t}$ be the $F$--transform of  $\sigma_{s,t}.$ Then there exists a family $(\nu_t)_{t\geq0}$ of probability measures such that \eqref{regularity} holds and $\varphi_{s,t}$ satisfies the Loewner equation \eqref{ODE2}.
\end{theorem}
\begin{proof}
We begin with the first part of the statement:\\
Each $\varphi_{s,t}$ is a univalent mapping from $\Ha$ into itself and can be represented as
 \begin{equation}
\label{star}\tag{$*$}\varphi_{s,t}(z)=z+\int_{\R}\frac{1}{u-z}\,\beta_{s,t}(du),
\end{equation}
where $\beta_{s,t}$ is a finite Borel measure with  $\beta_{s,t}(\R)=t-s$, see \cite{MR1201130}, Theorem 4 and the definition of the class $\mathcal{P}$ on p.1210.
From \cite[Proposition 2.2]{MR1165862} it follows that $\varphi_{s,t}$ is the $F$--transform of a probability measure $\sigma_{s,t}$ which has mean 0 and variance  $\beta_{s,t}(\R)=t-s$.\\
Because of \eqref{evol}, the conditions (a) and (b) in Def. \ref{LME} are satisfied.\\
Furthermore, we have
\begin{equation}\label{regularity2}|\varphi_{s,u}(z) - \varphi_{s,t}(z)| \leq \frac{|t-u|}{\Im(z)},\end{equation} 
see \cite[p.1214]{MR1201130}. By Theorem 2.5 in \cite{MR1165862}, we conclude that $\sigma_{s,u}$ 
converges weakly to $\sigma_{s,t}$ as $u\to t,$ i.e. condition (c) holds as well. \\

Next, let $\sigma_{s,t}$ be a normal monotone evolution family and let $\varphi_{s,t}$ be the $F$--transform of $\sigma_{s,t}.$ As $\sigma_{s,t}$ has mean 0 and finite variance, the mapping $\varphi_{s,t}$ and can be represented as \eqref{star} with $\beta_{s,t}(\R)=t-s$, \cite[Proposition 2.2]{MR1165862}. Condition (a) and (b) imply that $\varphi_{s,t}$ satisfies \eqref{evol}. Furthermore, the weak convergence, condition (c), implies that $\varphi_{s,u}$ converges locally uniformly to $\varphi_{s,t}$ as $u\to t$. From \cite[Theorem 3]{MR1201130}, it follows that $\varphi_{s,t}$ satisfies \eqref{ODE2}.
\end{proof}

\begin{remark} The measures $(\sigma_{t})_{t\geq0}:=(\sigma_{0,t})_{t\geq0}$ form a monotone semigroup, i.e. $\sigma_{t+s}=\sigma_t \rhd \sigma_s$, if and only if $\nu_t$ does not depend on $t.$ In this case, $\sigma_{s,t} = \sigma_{0,t-s}$, and so the whole evolution family is reduced to a semigroup.
\end{remark}
\begin{remark} Consider a monotone evolution family which is not necessarily normal. Under a suitable assumption on absolute continuity of $t\mapsto \varphi_{s,t}(z)$, the
 $F$--transforms $\varphi_{s,t}$ will now satisfy the differential equation
\begin{equation*}
\frac{\partial}{\partial t}\varphi_{s,t} = -P_t(\varphi_{s,t}) \qquad \text{for a.e. $t\in[s, \infty)$},\quad  \varphi_{s,s}(z)=z\in\Ha,
\end{equation*}
where $P_t$ has the form $$P_t = \gamma_t + \int_\R\frac{1+xz}{z-x} \tau_t(dx),$$
and $\gamma_t\in\R$ and $\tau_t$ is a positive finite measure for a.e. $t\geq0.$ This can be proven similarly and the differential equation is the time-dependent version of the monotone L\'{e}vy-Khintchine formula, see \cite[Section 4]{MR1853184}.\\
In general, however, differentiability almost everywhere does not hold: Let $\varphi_{s,t}(z) = z + C(t)-C(s),$ where $C:[0,\infty)\to \mathbb{R}$ is any continuous function. 
These functions are the F--transforms of the monotone evolution family 
$\{\delta_{C(s)-C(t)}\}_{0\leq s\leq t}$, but
in general, $t\mapsto \varphi_{s,t}(z)$ is not \emph{differentiable almost everywhere}.
\end{remark}

To summarize, for normal monotone evolution families we have three equivalent objects:
$$(\sigma_{s,t})_{0\leq s\leq t}\qquad  \underset{\text{F--transform}}{\longleftrightarrow} \qquad(\varphi_{s,t})_{0\leq s\leq t}  \qquad\underset{\text{Equation \eqref{ODE2}}}{\longleftrightarrow} \qquad(\nu_t)_{0\leq t}. $$

\begin{remark} There exists also a \emph{multiplicative} monotone convolution, see \cite{ber} and the references therein, which naturally corresponds to the so-called radial Loewner equation. This differential equation has been considered originally by C. Loewner in \cite{MR1512136}. 
\end{remark}

\subsection{Anti-monotone evolution families} 

Quite similarly, one defines anti-monotone independence and anti-monotone convolution, see \cite{MR1974583}. For probability measures $\mu, \nu$, the \emph{anti-monotone convolution} $\mu \lhd \nu$ is defined by $F_{\mu \lhd \nu} = F_{\nu} \circ F_{\mu}.$

\begin{definition}
We call a collection $(\sigma_{s,t})_{0\leq s\leq t}$ of probability measures a \emph{(normal) anti-monotone evolution family} if it satisfies the conditions of Definition \ref{LME} with (b) replaced by
\begin{itemize}
\item[(b')] $\sigma_{s,t} = \sigma_{u,t} \lhd \sigma_{s,u}$ whenever $0\leq s\leq u\leq t$.
\end{itemize}
\end{definition}

The family $(\varphi_{s,t})_{0\leq s\leq t}:=(F_{\sigma_{s,t}})_{0\leq s\leq t}$ is now what is called a \emph{reverse evolution family} in Loewner theory, see \cite[Definition 1.9]{MR3184323}, and one obtains analogously to Theorem \ref{Sonntag}:

\begin{theorem} Let $t>0$
and let $(\nu_s)_{0\leq s \leq t}$ be a family of probability measures satisfying \eqref{regularity} and let $M_s$ be the Cauchy transform of $\nu_s$. Denote by $\varphi_{s,t}$ the solution to 
\begin{equation}\label{ODE3}\frac{\partial}{\partial s}\varphi_{s,t} = M_s(\varphi_{s,t}) \qquad \text{for a.e. $s\in[0,t]$},\quad \varphi_{t,t}(z)=z\in\Ha.
\end{equation}
Then $\varphi_{s,t}$ is the $F$--transform of a measure $\sigma_{s,t},$ and  $(\sigma_{s,t})_{0\leq s\leq t}$ is a (part of a) normal anti-monotone evolution family.\\
Conversely, let $(\sigma_{s,t})_{0\leq s\leq t}$ be a normal anti-monotone evolution family and let $\varphi_{s,t}$ be the $F$--transform of $\sigma_{s,t}.$ For every $t>0$, there exists a family $(\nu_s)_{0\leq s \leq t}$ of probability measures such that \eqref{regularity} holds and $\varphi_{s,t}$ satisfies the (reverse) Loewner equation \eqref{ODE3}.
\end{theorem}
\begin{proof}
We consider the second statement first. Let $\varphi_{s,t}$ be the $F$--transforms of a normal anti-monotone evolution family $\sigma_{s,t}.$ First, note that $\sigma_{s,t}$ is also continuous w.r.t. to the variable $s.$ This follows from writing 
$\varphi_{s,t}=\varphi_{0,s}^{-1}\circ \varphi_{0,t}$ and using \cite[Theorem 2.5]{MR1165862}.\\
Let $T\geq t.$ It is now easy to see that $\{\sigma_{T-v, T-u}\}_{0\leq u\leq v\leq T}$ is (a part of) a normal monotone evolution family 
and we obtain equation \eqref{ODE3} by applying Theorem \ref{Sonntag} to differentiate w.r.t. to $v$ (and then changing $T-v$ to $s$, $T-u$ to $t$).\\
Now we consider the converse statement. We can first solve equation \eqref{ODE3} (see \cite[Theorem 4.2 (i)]{MR3184323}) and then reverse the time again to obtain the  
$F$--transforms of a normal monotone evolution family by Theorem \ref{Sonntag}. This implies that $\{\varphi_{s,t}\}_{0\leq s \leq t}$ corresponds to a normal anti-monotone
evolution family.
\end{proof}

\begin{remark} The chordal Loewner equation \eqref{ODE1} differs from \eqref{ODE3} only in 
the initial value, i.e. $\varphi_{t,t}(z)=z$
instead of $\varphi_{0,t}(z)=z.$\\
Equivalently (see \cite[Theorem 4.2 (ii)]{MR3184323}), one can describe $\varphi_{s,t}$ as the solution to
\begin{equation}\label{PDE2}\frac{\partial}{\partial t}\varphi_{s,t}(z) = -\frac{\partial}{\partial z}\varphi_{s,t}(z)\cdot M_t(z) \qquad \text{for a.e. $t\in[s,\infty)$}, \quad \varphi_{s,s}(z)=z\in\Ha,
\end{equation}
where $M_t$ is again the Cauchy transform of a probability measure $\nu_t.$
This equation basically follows by taking the derivative w.r.t $t$ in the
relation $\varphi_{s,t}\circ \varphi_{t,u}=\varphi_{s,u},$ $0\leq s\leq t\leq u$,
and using \eqref{ODE3}. \\
Note that \eqref{PDE2} is nothing but \eqref{PDE1}.
\end{remark}

\subsection{The slit equation}

The most prominent Loewner equation is the so-called slit equation, which simply corresponds to  $\nu_t=\delta_{U(t)}$, where $U:[0,\infty)\to\R$ is a continuous function. Both  equations, \eqref{ODE2} and \eqref{ODE3}, are called slit equations in this case.\\ 

Let us stay now in the setting of monotone probability theory. Equation \eqref{ODE2}  is given by

\begin{equation}\label{slit}\frac{\partial}{\partial t}\varphi_{s,t} = \frac{1}{U(t)-\varphi_{s,t}} \qquad \text{for a.e. $t\in[s, \infty)$}, \quad \varphi_{s,s}(z)=z\in\Ha.\end{equation}

If the so called driving function $U$ is smooth enough, then the solutions $\varphi_{s,t}$ are conformal mappings of the form $\varphi_{s,t}:\Ha\to \Ha\setminus \gamma_{s,t},$ where $\gamma_{s,t}$ is a simple curve $\gamma_{s,t}:[s,t]\to\overline{\Ha}$ with $\gamma_{s,t}(s)=U(s)\in\R$ and $\gamma_{s,t}(s,t]\subset\Ha.$ Such a curve is also called a \emph{slit} of the upper half-plane. For the smoothness conditions, we refer to \cite{MarshallRohde:2005, Lind:2005, LindMR:2010}. \\
Conversely, for every slit $\gamma$ there exists $T>0$ and $U:[0,T]\to\R$ such that the solution of \eqref{slit} satisfies $\varphi_{0,T}(\Ha) = \Ha\setminus \gamma$; see \cite{GM13} and the references therein.

 If $U(t)\equiv u\in\R,$ then the solution $\varphi_{s,t}=\varphi_{0,t-s}$ to \eqref{slit} is given by $$\varphi_{0,t}(z)=u+\sqrt{(z-u)^2-2t},$$ which maps $\Ha$ onto $\Ha$ minus a straight line segment from $u$ to $u+i\sqrt{2t}.$
The corresponding probability measure is given by $\sigma_{0,t}=\delta_{-u} \rhd \mu_{A,t}\rhd \delta_{u}.$\\

 If $U(t)$ is not constant, then one can approximate $\varphi_{0,t}$ by the solution of a piecewise constant driving function.\\
Choose $N\in\N$ and let $\Delta t=\frac1{N}$ be a time interval.  Assume we are interested in $\varphi_{0,K\Delta t},$ $K\in\N.$ Approximately, it can be obtained as follows: Let $\Delta_0, \Delta_1,...$ be defined by $\Delta_k = U((k+1)\Delta t)-U(k\Delta t).$ We have $\sigma_{k\Delta t, (k+1)\Delta t}\approx \delta_{-\Delta_k} \rhd \mu_{A,\Delta t}\rhd \delta_{\Delta_k},$ and consequently
\begin{equation}\label{rev_SLE_approx}
\sigma_{0,K\Delta t} = \sigma_{(K-1)\Delta t,K\Delta t} \rhd \sigma_{0,(K-1)\Delta t} = ... \approx
\mathop{\rhd}\limits_{k=0}^{K-1} \left(\delta_{-\Delta_k} \rhd \mu_{A,\Delta t}\rhd \delta_{\Delta_k}\right)=: \sigma^N_{0,K\Delta t}.
\end{equation}
We note that for the computation of the conformal mappings, a slightly different approximation is more suitable for practical use, see \cite{MR2348786}.

\begin{example}[Schramm-Loewner Evolution] Let $B_t:[0,\infty)\to\R$ be a standard Brownian motion and $\kappa\in(0,\infty).$ Let $\nu_t=\delta_{\sqrt{\kappa/2}B_t}.$ The solution $g_t$ to the stochastic differential equation \eqref{ODE1}, i.e.
$$\frac{\partial}{\partial t}g_t = \frac{1}{g_t - \sqrt{\kappa/2}B_t}, \quad g_0(z)=z,$$
describes the growth of a random curve in $\Ha$ from $0$ to $\infty$, which is called Schramm-Loewner evolution (SLE). This curve is a slit with probability one if and only if $\kappa\in(0,4]$. SLE and its generalizations have important applications in statistical mechanics and probability theory. We refer to \cite{Lawler:2005} for an introduction.\\

The solution to \eqref{slit} with $U(t)=\sqrt{\kappa/2}B_t$ is called backward SLE (see \cite{MR3477781}). It corresponds to a
 (classically) random normal monotone evolution family $(\sigma_{s,t})_{0\leq s\leq t}.$\\
Now we can approximate $\sigma_{0,K\Delta t}$ as follows: Let $\Delta_0, \Delta_1,...$ be a sequence of (classically) independent normally distributed  random variables with mean 0 and variance $\frac{\kappa}{2}\Delta t$. Then
\begin{equation*}
\sigma_{0,K\Delta t} \approx
\mathop{\rhd}\limits_{k=0}^{K-1} \left(\delta_{-\Delta_k} \rhd \mu_{A,\Delta t}\rhd \delta_{\Delta_k}\right) =: \sigma^N_{0,K\Delta t}.
\end{equation*}
We have $\sigma_{0,K\Delta t}=\lim_{N\to\infty}\sigma^N_{0,K\Delta t}$ in the sense of convergence in distribution with respect to the topology induced by weak convergence; see \cite{MR3409694} for an even stronger statement.
\end{example}

\subsection{Free evolution families}

Let $(\mu_t)_{t\geq0}$ be a free semigroup, i.e. $\mu_{t+s}=\mu_t \boxplus \mu_s.$ In this case, $R_\mu$ has an analytic extension to $\Ha$ and
\begin{equation}\label{not_normal}
R_\mu = \alpha + \int_\R \frac{z+x}{1-xz}\,\nu'(dx)
\end{equation}
with $\alpha\in\R$ and $\nu'$ is a finite positive measure. Moreover, $\mu$ has mean 0 and finite variance $\sigma^2$ if and only if
$$R_\mu = \int_\R \frac{z}{1-xz}\,\nu(dx)= \int_\R \frac{1}{1/z-x}\, \nu(dx) = G_{\nu}(1/z),$$
where $\nu$ is a  measure with $\nu(\R)=\sigma^2;$ see \cite[Theorem 6.2]{MR1165862} and \cite[Theorem 5.10]{MR1254116}, \cite[Section 4.1]{MR3129803}. \\

By generalizing equation \eqref{semigroup}, we obtain evolution families with respect to the free convolution; see \cite{MR1605393}.

\begin{definition}
We call a collection $(\sigma_{s,t})_{0\leq s\leq t}$ of probability measures a \emph{(normal) free evolution family} if it satisfies the conditions of Definition \ref{LME} with (b) replaced by
\begin{itemize}
\item[(b'')] $\sigma_{s,t} = \sigma_{u,t} \boxplus \sigma_{s,u}$ whenever $0\leq s\leq u\leq t$.
\end{itemize}
\end{definition}

Let $\nu_t$ be a family of probability measures  such that $t\mapsto G_{\nu_t}(z)$ is measurable for every $z\in\Ha.$ 

Now we consider the non-autonomous version of equation \eqref{semigroup} with $\alpha=\delta_0$ and we replace $R(z)$ by $G_{\nu_t}(1/z)$:
\begin{equation*}
\frac{\partial}{\partial t}G_{s,t}(z) = -\frac{\partial}{\partial z}G_{s,t}(z)\cdot G_{\nu_t}(1/G_{s,t}(z))\quad  \text{for a.e.}\;  t\in[s,\infty), \quad G_{s,s}(z) = \frac1{z}. \end{equation*}
The $R$--transform in free probability theory corresponds to the $F$--transform in monotone probability theory. So, instead, we take $R_{s,t}=G_{s,t}^{-1}-\frac1{z}$ and obtain the simple equation
\begin{equation}\label{semigroup2}
\frac{\partial}{\partial t}R_{s,t}(z) = G_{\nu_t}(1/z)\quad  \text{for a.e.}\; t\in[s,\infty), \quad R_{s,s}(z) = 0. \end{equation}

\begin{theorem} Under the above assumptions, \eqref{semigroup2} has a unique solution $R_{s,t}$. For all $0\leq s\leq t,$ $R_{s,t}$ is the $R$--transform of a probability measure $\sigma_{s,t}$ and the collection $(\sigma_{s,t})_{0\leq s\leq t}$ is a normal free evolution family.\\
Conversely,  let $(\sigma_{s,t})_{0\leq s\leq t}$ be a normal free evolution family. Then there exists a family $(\nu_t)_{t\geq0}$ of probability measures such that $G_{\nu_t}$ satisfy \eqref{regularity} and the $R$--transform $R_{s,t}$ of $\sigma_{s,t}$ satisfies equation \eqref{semigroup2}.
\end{theorem}
\begin{proof} Obviously, the solution $R_{s,t}$ of \eqref{semigroup2} is simply given by 
$$R_{s,t}(z) = \int_s^t \int_\R \frac{1}{1/z-x}\, \nu_\tau(dx) d\tau.$$
As $R_{s,t}$ is a holomorphic mapping with $R_{s,t}(\Ha)\subset \overline{\Ha},$ it is easy to see that this function also has the form $R_{s,t}(z)=G_{\alpha_{s,t}}(1/z)$ for a positive measure $\alpha_{s,t};$ see \cite[Lemma 1]{MR1201130}. The behaviour of $R_{s,t}$ for $z$ near $0$ yields that $\alpha_{s,t}(\R)=\int_s^t \nu_\tau(\R) \,d\tau = t-s$. This implies that $R_{s,t}(z)$ is the $R$--transform of a probability measure $\sigma_{s,t}$ with mean 0 and variance $t-s.$ Clearly, $R_{u,t} + R_{s,u}= R_{s,t}$ whenever $0\leq s\leq u\leq t$ and $R_{t,t}=0,$ which implies  $\sigma_{s,t} = \sigma_{u,t} \boxplus \sigma_{s,u}$, $\sigma_{t,t}=\delta_0.$ Furthermore, as $t\mapsto R_{s,t}$ is continuous with respect to locally uniform convergence, we obtain from \cite[Proposition 5.7]{MR1254116} that $\sigma_{s,u}$ converges weakly to $\sigma_{s,t}$ as $u\to t$.\\

Conversely, let $R_{s,t}$ be the $R$--transform of $\sigma_{s,t}.$ Fix some $s\geq0$ and let $t\geq s.$ \\
Then, $R_{s,t}(z)=G_{\alpha_{s,t}}(1/z)$ with $\alpha_{s,t}(\R)=t-s.$ This implies that, for $z\in\Ha,$ the map $t\mapsto R_{s,t}(z)$ is Lipschitz continuous:
$$|R_{s,t}(z) - R_{s,u}(z)| = |R_{u,t}(z)| \leq \int_{\R}\frac1{|1/z-u|} \alpha_{u,t}(du)\leq \frac{t-u}{\Im(1/z)},$$
for all $s \leq u \leq t.$\\
Thus, $t\mapsto R_{s,t}(z)$ is differentiable for every $t\in[s,\infty)$ except a zero set $\mathcal{N}(z).$ By considering a countable dense subset $D\subset \Ha,$ we conclude that there exists a zero set  $\mathcal{N}\subset [s,\infty)$ such that $t\mapsto R_{s,t}(z)$ is differentiable for every $z\in D$ and every $t\in[s,\infty)\setminus \mathcal{N}.$\\  
 Now assume $t\mapsto R_{s,t}$ is differentiable at $t_0$ for all $z\in D$ and let $h>0$. Then $h^{-1}(R_{s,t_0+h}-R_{s,t_0})=h^{-1} R_{t_0,t_0+h}.$ This function can be represented as $h^{-1} R_{t_0,t_0+h}(z) = G_{\beta_{t_0,h}}(1/z)$ for a positive measure $\beta_{t_0,h}$ with $\beta_{t_0,h}(\R) = h^{-1} (t+h - t) = 1.$\\
It can easily be verified that the closure of the set of all Cauchy-transforms of probability measures is the set of all Cauchy-transforms of non-negative measures with mass $\leq 1.$ This family is locally bounded as every such $G_\nu$ satisfies $|G_\nu(z)| \leq \frac{1}{\Im(z)}.$\\
We assumed that the limit $\lim_{h\to0}G_{\beta_{t_0,h}}$ exists for all $z\in D.$ By the the Vitali-Porter theorem, see \cite{MR1211641}, Section 2.4, we have in fact locally uniform convergence and thus $$\lim_{h\to0}G_{\beta_{t_0,h}}(z) = G_{\nu_{t_0}}(z)$$ for a non-negative measure $\nu_{t_0}$ with $\nu_{t_0}(\R)\leq 1.$ In particular, $t\mapsto R_{s,t}(z)$ is differentiable for all $t\in[s,\infty)\setminus \mathcal{N}$ and all $z\in\Ha.$ \\
By the proof of the first part, we have that $\int_s^t \nu_{\tau}(\R)\,d\tau$ is equal to $t-s$; hence $\nu_{\tau}(\R)=1$ for a.e. $\tau \geq s.$ Clearly, we can choose $(\nu_t)_{t\geq0}$ such that $\nu_t$ is a probability measure for every $t\geq0$ and that $t\mapsto G_{\nu_t}(z)$ is measurable for every $z\in\Ha.$
\end{proof}

Thus, a normal free evolution family can be described as:
$$(\sigma_{s,t})_{0\leq s\leq t}\qquad  \underset{\text{R--transform}}{\longleftrightarrow} \qquad(R_{s,t})_{0\leq s\leq t}  \qquad\underset{\text{Equation \eqref{semigroup2}}}{\longleftrightarrow} \qquad(\nu_t)_{0\leq t}. $$

\begin{remark} If we consider a free evolution family, not necessarily normal, then, under a suitable assumption of absolute continuity, the $R$--transforms $R_{s,t}$ correspond to the differential equation
\begin{equation*}
\frac{\partial}{\partial t}R_{s,t} = P_t(R_{s,t}) \qquad \text{for a.e. $t\in[s, \infty)$},\quad R_{s,s}(z)=z\in\Ha,
\end{equation*}
where $P_t$ has the form \eqref{not_normal} with some $\alpha_t\in\R $ and a positive finite measure $\nu'_t$ for a.e. $t\geq0.$
\end{remark}

\begin{example}[``free slit equation''] One can look at the analogue of the slit equation in the free setting in two different ways. First, let $U_t:[0,\infty)\to\R$ be a continuous function and consider $\nu_t=\delta_{U_t},$ i.e. $G_{\nu_t}(z)=\frac1{z-U(t)}.$ Then $R_{s,t}(z) = \int_s^t \frac1{\frac1{z}- U(\tau)} \,d\tau$. \\
Secondly, we look at the analogue of \eqref{rev_SLE_approx} in the free setting, i.e. we replace the arcsine distribution by the semicircle distribution. However, as $\delta_{-\Delta_k} \boxplus \mu_{W,\Delta t}\boxplus \delta_{\Delta_k}=\mu_{W,\Delta t}$ due to commutativity of $\boxplus$, we simply obtain that $\sigma^N_{0,K\Delta t}=\sigma_{0,K\Delta t}=\mu_{W,K\Delta t}$, i.e. we obtain a free Brownian motion $\sigma_{s,t}=\mu_{W,t-s}.$
\end{example}

\section{Further Remarks}\label{realization}

\begin{question} Let $\mu$ be a probability measure such that its $F$--transform $F_\mu$ is injective and $F_{\mu}(\Ha)=\Ha \setminus \gamma$ for a slit $\gamma.$ How can those probability measures $\mu$ be characterized? 
\end{question} 

A basic property of those probability measures is the symmetry with respect to a point $u\in\R$, which is the preimage of the tip of the slit $\gamma$ with respect to the mapping $F_\mu.$ 

\begin{proposition} Let $\mu$ be a probability measure such that $F_\mu$ maps $\Ha$ conformally onto $\Ha\setminus \gamma$, where
$\gamma$ is a slit.
\begin{itemize}
 \item[(a)] Assume $\gamma$ is starting at $0$. Then $\supp \mu$ is a compact interval $[a,b]$ and  
 $\mu$ has a density $d(x)$ on $(a,b).$ 
 \item[(b)] Assume $\gamma$ is starting at $s\in \R\setminus\{0\}.$ Then $\supp \mu= \{x_0\}\cup[a,b]$, where $\mu$ 
 has a density $d(x)$ on the compact interval $[a,b]$ and an atom at some $x_0\in \R\setminus [a,b].$  
\end{itemize}
In both cases, there exists $u\in(a,b)$ and a homeomorphism $h:(a,b)\to (a,b)$ with $h(u)=u,$ $h(a,u]=[u,b)$ such that 
$d(h(x)) = d(x)$ for all $x\in(a,b).$ 
\end{proposition}
\begin{proof}
As the domain $\Ha\setminus \gamma$ has a locally connected boundary, the mapping $F_\mu$ can be extended continuously to $\overline{\Ha};$ see \cite[Theorem 2.1]{MR1217706}.\\
There exists an interval $[a,b]$ such that $f([a,b])=\gamma$ and there is a unique $u\in (a,b)$ such that $F_\mu(u)$ is the tip of the slit. All points $[a,u]$ correspond to the left side, all points $[u,b]$ to the right side of $\gamma.$ (This orientation follows from the behaviour of $F_\mu(x)$ as  $x\to\pm\infty.$) Hence, there exists a unique homeomorphism $h:(a,b)\to(a,b)$ with $h(u)=u,$ $h(a,u]=[u,b)$ such that $F_\mu(h(x))=F_\mu(x)$ for all $x\in(a,b).$\\
 It follows from \cite[Theorem F.6]{MR2953553} that $\mu$ is absolutely continuous on $(a,b)$ and the density $d(x)$ satisfies $d(x) = \frac{-1}{\pi}\Im(1/F_\mu)(x).$ Hence, $d(h(x)) = d(x)$ for all $x\in(a,b).$\\
 If $0$ is not the starting point of the slit, $\mu$ is absolutely continuous on $[a,b]$ as $G_\mu(a), G_\mu(b)\not=\infty$. 
 Furthermore, $F_\mu$ has exactly one zero $x_0\in\R\setminus[a,b]$ in this case. Hence, 
 $1/F_\mu$ has a pole at $z=x_0$ and we conclude that $\mu$ has an atom at $x_0$ (see \cite[Theorem F.2]{MR2953553}). 
 As $\Im(1/F_\mu)(x)=0$ for all $x\in\R\setminus(\{x_0\}\cup[a,b]\},$ we conclude that $\supp \mu = \{x_0\}\cup[a,b]$, 
 again by using \cite[Theorem F.6]{MR2953553}.\\
 If $0$ is the starting point of $\gamma$, then $\supp \mu = [a,b]$. 
\end{proof}

\begin{remark}
A slit $\gamma$ is called \emph{quasislit} if $\gamma$ approaches $\R$ nontangentially and $\gamma$ is the image of a line segment under a quasiconformal mapping.

The theory of conformal welding implies: $\gamma$ is a quasislit if and only if $h$ is quasisymmetric; see \cite[Lemma 6]{Lind:2005} and \cite[Lemma 2.2]{MarshallRohde:2005}.
\end{remark}

 Let $(\sigma_{s,t})_{0\leq s\leq t}$ be a (free/monotone/anti-monotone) evolution family of compactly supported probability measures. Of course, one is interested in realizations of such a family of distributions as a process on a $C^*$--algebra.

\begin{definition}
 A \emph{realization} of $(\sigma_{s,t})_{0\leq s \leq t}$ is a $C^*$--algebra $(\mathcal{A}, \varphi)$ with a collection $(X_{t})_{0 \leq t}\subset \mathcal{A}$ of self-adjoint random variables such that 
\begin{itemize}
\item[(a)] $X_0=0$,
\item[(b)] the distribution of $X_{t}-X_s$ is given by $\sigma_{s,t}$  for all $0\leq s \leq t,$
\item[(c)] the increments $(X_{t_{2}}-X_{t_{1}}, ..., X_{t_n}-X_{t_{n-1}})$ are independent for all $0\leq  t_1\leq t_2\leq ... \leq t_n.$
\end{itemize}
\end{definition}
One can require also further regularity conditions for the map $t\mapsto X_t$.\\

A realization of the free semigroup $(\sigma_{s,t})_{0\leq s\leq t} = \mu_{W,t-s}$ is called a \emph{free Brownian motion}. Similarly, a realization of the monotone semigroup $(\sigma_{s,t})_{0\leq s\leq t} = \mu_{A,t-s}$ is called a \emph{monotone Brownian motion},  which corresponds to Example \ref{line_segment}.\\
In general, one can switch between non-commutative and classical evolution families of probability measures by using the L\'{e}vy--Khintchine  representation formulas; see \cite[Remark 5.17, Theorem 4.14]{MR2213451} and \cite[Theorem 2.2]{MR2230621}.\\

A free Brownian motion can be realized on the free Fock space (see \cite{MR1030725})

$$F(L^2(\R)) = \Omega \C \oplus \bigoplus_{n=1}^\infty L^2(\R^n),$$ 
where $\Omega\in L^2(\R)$ has norm 1.\\ Now one takes $\mathcal{A}=\mathcal{B}(F(L^2(\R)))$, which is the space of all bounded linear operators on $F(L^2(\R)) ,$ and  $\varphi:\mathcal{A}\to\C$, $\varphi(a) = \left< a \Omega, \Omega\right>.$ \\

In a similar way, one can realize a monotone Brownian motion on the monotone Fock 
space; see \cite{MR1462227}. Realizations can also be described by 
``quantum stochastic differential equations''. 
We refer to \cite[Theorem 4.1]{MR2230621} and \cite[p.246]{MR2213451} for the monotone and \cite[Section 5.4 on p.121 and Section 6 on p.123]{MR2213451} for the free case.

\begin{question} Is it possible to realize the (classically) random monotone/anti-monotone evolution family of SLE (i.e. $\nu_t=\delta_{\sqrt{\kappa/2}B_t}$)? 
\end{question}


\def\cprime{$'$}

%
%
%
%
%
\end{document}